\documentclass[leqno,11pt,twoside,draft]{article}
\usepackage{pliska}
\usepackage[intlimits]{amsmath}
\usepackage{amsfonts}
\usepackage{amssymb}
\usepackage{graphicx}
\usepackage{caption2}
\usepackage{graphics}

\newcommand{\Pb}{\mathrm{P}}

\begin{document}\label{begin}
\DeclareGraphicsExtensions{.jpeg}
\head 
{
} {Multivariate compounds with equal number of summands   
} {  Pavlina K. Jordanova       \footnote{The author is partially supported by Project RD-08-69/02.02.2016 from the Scientific Research Fund in
Konstantin Preslavsky University of Shumen, Bulgaria.
}           
} {Multivariate compounds with equal number of summands     
} {Pavlina K. Jordanova
}
{.....
}
{....
}

\begin{abstract}
The paper considers multivariate discrete random sums with equal number of summands.
Such distributions describe    the total claim amount received by a company in a fixed time point. In Queuing theory they characterize cumulative waiting times of customers up to time t. In Theory of branching processes they model the number of heirs at a fixed point in time.

Here some general properties and formulae for numerical characteristics of these distributions are derived and some particular cases are considered.
\end{abstract}

\kwams{62E15}
  {Compound distributions; Multinomial distribution; Negative multinomial distribution}

\section{Introduction}

Univariate random sums, are considered intensively in the beginning of the last century. Abraham Wald \cite{Wald} found the relation between the numerical characteristics of these random sums, their summands and their number of summands. Nowadays his results are well known as "Wald equalities".  Investigation of the particular cases starts with Poisson distributed number of summands and geometrically distributed on positive integers summands. This case is related with the names of Alfred Aeppli (1924) and George P$\acute{o}$lya (1930). An exploration of univariate compound Poisson processes with their applications in insurance can be seen e.g. in the works of O. Lundberg \cite{Lundberg} (1940).

Here we consider the corresponding general multivariate distributions. Along the paper "$\stackrel{\rm d}{=}$" stands for the equality in
distribution of two random variables (r.vs.), "$\equiv$" for coincidence of two classes of probability laws, $FI\, X : = \frac{Var\,X}{E\,X}$ is the Fisher index of dispersion, $CV (X, Y) = \sqrt{\frac{cov(X, Y)}{EX\, EY}}$,  and $CV \, X = \frac{\sqrt{Var\,X}}{E\,X}$ is the coefficient of variation of the distribution of the random variable (r.v.) $X$. The sign "$\sim$" expresses the fact, that the distribution of the corresponding  r.v. belongs to a given class of probability laws.

Let $N$ be a non-negative, integer valued r.v. with probability generating function (p.g.f.) $G_N$. Suppose $(Y_{1i}, Y_{2i}, ..., Y_{ki})$, $i = 1, 2, ...$ are independent identically distributed (i.i.d.) random vectors, defined on the same probability space $(\Omega, \mathcal{A}, \Pb)$. Denote their p.g.f. by $G_{Y_{1i}, Y_{2i}, ..., Y_{ki}}$, $i = 1, 2, ...$. The random vector $(X_{1N}, X_{2N},..., X_{kN})$, which coordinates are defined by
\begin{equation}\label{Def1}
X_{sN} = I_{N > 0} \sum_{i=1}^N Y_{si}, \quad s = 1, 2, ..., k
\end{equation}
is called {\bf  compound $N$($N$-stopped sum, random sum) with equal number of  $Y$ summands}. In any particular case, the letters $N$ and $Y$, will be replaced by the name of the corresponding distribution. Briefly we will denote this in the following way

$$(X_{1N}, X_{2N},..., X_{kN}) \sim C\,N\,Y(\vec{a}_N; \vec{b}_Y).$$
Here the letters $C\,N\,Y(\vec{a}_N; \vec{b}_Y)$ have the following meaning: $C$ comes from "com\-pound"$,$ $N$ is the abbreviation of the distribution of the number of summands, with parameters $\vec{a}_N$ and $Y$ is the abbreviation of the distribution of the vector of summands, with parameters $\vec{b}_Y$.

  The cases when $N$ is Poisson distributed and $(Y_{1i}, Y_{2i}, ..., Y_{ki})$ are one of the
\begin{itemize}
  \item Multinomial(Mn),
  \item Negative Multinomial(NMn),
  \item Multivariate Geometric(MGe),
  \item Multivariate Poisson (MPo)
\end{itemize}
distributed are investigated in 1962, by G. Smith \cite{Smith}. The particular case, when the summands are MGe on positive integers and the number of summands are Poisson, is investigated in series of papers of Minkova et. al, see e.g.  \cite{LedaBala} and \cite{LedaBala1}. In 1971, Khatri \cite{Khatri} mentioned that the class of these distributions is too wide and investigates some multivariate particular cases. In 1981 Charalambides \cite{Charalambides}, together with Papageorgiou \cite{ChPapageorgiou} explored bivariate case. Charalambides \cite{Charalambides} has called all distributions with p.g.f. $G_N(G_{Y_{11}, Y_{21}}(z_1, z_2))$ "generalized", and has mentioned that many bivariate contagious and compound distributions belong to this class. In 1996 Wang \cite{Wang}, obtained multivariate compound Poisson distributions  as the limiting distributions of multivariate sums of independent random vectors. He called the distribution, defined in (\ref{Def1}),
"multivariate compound distribution of type I". The covariances formula for the general case of these distributions
\begin{equation}\label{covRaluca}
cov(X_{iN}, X_{jN}) = EN cov(Y_{i1}, Y_{j1}) + E Y_{i1} E Y_{j1} Var N
\end{equation}
can be found in Sundt et. al \cite{Sundt}. In series of papers they consider recursions formulae for these distributions.

This paper develops the results in Jordanova \cite{Jordanova2016}. Their proves are presented and new explicit relations between the probability mass functions (p.m.fs.) and the mean square regressions in particular cases are obtained. After the characterisation, some simulations of realizations of random vectors with such distributions are made, and the dependence between their coordinates is visualised. These observations show that these distributions are appropriate for modelling of both: linear dependence between the coordinates and clustering in the observations. They can serve to overcome Simpson's paradox, described e.g. in Colin \cite{Colin}. The main machinery of probability generating functions (p.g.fs.)  used here can be seen e.g. in the book of Johnson et. al \cite{bib3}.

In the next section we investigate the general properties of these distributions without specifying the probability laws of the summands in (\ref{Def1}), or those of the number of summands. Then, in Section 3, the particular cases of Mn and NMn summands are considered. The paper finishes with some conclusive remarks.

\section{General properties of multivariate compound distributions}

In this section we investigate the main properties of the distribution of the multivariate random sums, defined in (\ref{Def1}), together with their numerical characteristics and conditional distributions.

\begin{theorem}\label{ThGeneral} Let $X_{1N}, X_{2N},..., X_{kN}$ be a random vector, defined in (\ref{Def1}).

Its distribution possesses the following properties.

\begin{enumerate}
  \item Characterisation of the joint probability distribution.

  If the p.g.fs. exist in $(z_1, z_2, ..., z_k)$, then
        \begin{equation}\label{PGFGeneral}
    G_{X_{1N}, X_{2N},..., X_{kN}}(z_1, z_2, ..., z_k) = G_N (G_{Y_{11}, Y_{21}, ..., Y_{k1}}(z_1, z_2, ..., z_k)).
    \end{equation}
  Its probability mass function (p.m.f.) satisfies the following equalities

  $P(X_{1N} = 0,  X_{2N} = 0, ..., X_{kN} = 0) = G_N(G_{Y_{11}, Y_{21}, ..., Y_{k1}}(0, 0, ..., 0)) = $
   \begin{equation}\label{pmf0}
    = G_N(P(Y_{11} = 0,  Y_{21} = 0, ..., Y_{k1} = 0)).
   \end{equation}
   If the following derivatives exist, then for $s = 1, 2, .., k$, $x_s = 0, 1, ...$:  $(x_1, x_2, ..., x_k) \not= (0, 0, ..., 0)$,

   $P(X_{1N} = x_1,  X_{2N} = x_2, ..., X_{kN} = x_k) = $
    \begin{equation}\label{pmf}
   = \frac{\partial^{x_1 + x_2 + ... + x_k}G_N (G_{Y_{11}, Y_{21}, ..., Y_{k1}}(z_1, z_2, ..., z_k))}{x_1!x_2!...x_k!\partial z_1^{x_1}\partial z_2^{x_2}...\partial z_k^{x_k}}|_{(z_1, z_2, ..., z_k) = (0, 0, ..., 0)} =
   \end{equation}
    \begin{equation}\label{pmffpv}
   = \sum_{n = 1}^\infty P(N = n) \sum_{i_{11} + i_{12} + ... + i_{1n} = x_1} \sum_{i_{21} + i_{22} + ... + i_{2n} = x_2} ... \sum_{i_{k1} + i_{k2} + ... + i_{kn} = x_k}
    \end{equation}
    $$\prod_{s = 1}^n  P(Y_{1s} = i_{1s}, Y_{2s} = i_{2s}, ..., Y_{ks} = i_{ks}).$$
   \item Characterisation of subsets of the coordinates.

   For all $r = 1, 2, ..., k$ and for all subset of coordinates $X_{i_1N}, X_{i_2N}, ..., X_{i_rN}$,
   \begin{equation}\label{pgfsubset}
    G_{X_{i_1N}, X_{i_2N},..., X_{i_rN}}(z_{i_1}, z_{i_2}, ..., z_{i_r}) = G_N (G_{Y_{11}, Y_{21}, ..., Y_{k1}}(\vec{z}_{i\overline{i}})),
    \end{equation}
  where the vector $\overrightarrow{z}_{i\overline{i}}$ has coordinates equal to $1$ on the places different from $i_1, i_2, ..., i_r$ and correspondingly coordinates $z_{i_1}, z_{i_2}, ..., z_{i_r}$ on the other places.

  \bigskip

  $P(X_{i_1N} = x_{i_1},  X_{i_2N} = x_{i_2}, ..., X_{i_rN} = x_{i_r}) = $
   \begin{equation}\label{pmfsubsets}
   = \frac{\partial^{x_{i_1} + x_{i_2} + ... + x_{i_r}}G_N (G_{Y_{11}, Y_{21}, ..., Y_{k1}}(\vec{z}_{i\overline{i}}))}{x_{i_1}!x_{i_2}!...x_{i_r}!\partial z_{i_1}^{x_{i_1}}\partial z_{i_2}^{x_{i_2}}...\partial z_{i_r}^{x_{i_r}}}|_{(z_{i_1}, z_{i_2}, ..., z_{i_r}) = (0, 0, ..., 0)}.
   \end{equation}
\item Marginal distributions and their dependence on $N$. For $i = 1, 2, ..., k$,
  \begin{equation}\label{pgfmarginal}
G_{X_{iN}}(z_{i}) = G_N (G_{Y_{11}, Y_{21}, ..., Y_{k1}}(1, 1, ..., 1, z_i, 1, ..., 1)),
\end{equation}
    $P(X_{iN} = 0) = G_N(P(Y_{i1} = 0)).$

    For $x_i \not= 0$,  $ P(X_{iN} = x_i) =$
   \begin{equation}\label{marginal}
    = \sum_{n = 1}^\infty P(N = n) \sum_{j_{1} + j_{2} + ... + j_{n} = x_i} P(Y_{i1} = j_{1})P(Y_{i1} = j_{2})...P(Y_{i1} = j_{n}).
    \end{equation}
   If $E\,Y_{i_1} < \infty$ and $E\,N < \infty$, then
\begin{equation}\label{expectation}
E\,X_{i_N} = E\,N\, E\,Y_{i_1}, \quad E\,(X_{i_N}|N = n) = n\,E\,Y_{i_1}.
\end{equation}
   If $E\,(Y_{i1}^2) < \infty$ and $E\,(N^2) < \infty$, then
\begin{equation}\label{veriance}
Var\, X_{iN} = Var\, N \, (E\,Y_{i_1})^2 + E\,N Var\, Y_{i_1}, \, FI\,X_{iN} =  FI\, N \, E\,Y_{i_1} + FI\, Y_{i_1}
\end{equation}
\begin{equation}\label{cov}
cov(X_{i_N}, N) = EY_{i_1}\, Var \,N,
\end{equation}
$cor(X_{i_N}, N) =$
\begin{equation}\label{cor}
= E\,Y_{i_1} \sqrt{\frac{Var\, N}{Var\, N \, (E\,Y_{i_1})^2 + E\,N Var\, Y_{i_1}}} = \sqrt{\frac{FI\, N}{FI\, N + (CV\, Y_{i_1})^2}}.
\end{equation}
 \item Conditional distributions.
     \begin{equation}\label{conditionals}
      P(X_{iN} = x_i| X_{jN} = x_j) = \frac{1}{x_i!}\frac{\frac{\partial^{x_i + x_j}G_N (G_{Y_{11}, Y_{21}, ..., Y_{k1}}(\vec{z}_{ij\overline{ij}}))}{\partial z_i^{x_i}\partial z_j^{x_j}}|_{(z_{i}, z_{j}) = (0, 0)}}    {\frac{\partial^{x_j}G_N (G_{Y_{11}, Y_{21}, ..., Y_{k1}}(1, ..., 1, z_j, 1, ..., 1))}{\partial z_i^{x_i}}|_{z_j = 0}},
      \end{equation}
    where the vector $\vec{z}_{ij\overline{ij}}$ has coordinates equal to $1$ on the places different from $i$, $j$ and correspondingly coordinates $z_i, z_j$, on the places  $i$ and $j$.
   \begin{equation}\label{conditionalpgf}
       G_{X_{iN}}(z_i| X_{jN} = x_j) =  \frac{\frac{\partial^{x_j}G_N (G_{Y_{11}, Y_{21}, ..., Y_{k1}}(\vec{z}_{ij\overline{ij}}))}{\partial z_j^{x_j}}}
   {\frac{\partial^{x_j}G_N (G_{Y_{11}, Y_{21}, ..., Y_{k1}}(1, ..., 1, z_j, 1, ..., 1))}{\partial z_j^{x_j}}}|_{z_j = 0},
   \end{equation}
   \begin{equation}\label{conditionalexpectation}
    E (X_{iN}| X_{jN} = x_j) = \frac{\frac{\partial}{\partial z_i}\left[\frac{\partial^{x_j}G_N (G_{Y_{11}, Y_{21}, ..., Y_{k1}}(\vec{z}_{ij\overline{ij}}))}{\partial z_j^{x_j}}|_{z_j = 0}\right]|_{z_i = 1}}
   {\frac{\partial^{x_j}G_N (G_{Y_{11}, Y_{21}, ..., Y_{k1}}(1, ..., 1, z_j, 1, ..., 1))}{\partial z_j^{x_j}}|_{z_j = 0}},
   \end{equation}
  \item Dependence between the coordinates.

    $cor(X_{iN}, X_{jN}) =$
    \begin{equation}\label{corelationX}
    = \frac{EN \, cov(Y_{i1}, Y_{i2}) + EY_{i1} EY_{j1} Var\, N}{\sqrt{[Var\, N \, (E\,Y_{i_1})^2 + E\,N Var\, Y_{i_1}][Var\, N \, (E\,Y_{j_1})^2 + E\,N Var\, Y_{j_1}]}} =
    \end{equation}
    $$= \frac{[CV\,(Y_{i1}, Y_{i2})]^2 + FI\, N}{\sqrt{[(CV\, Y_{i1})^2+ FI\, N][(CV\, Y_{j1})^2+ FI\, N]}}.$$
   \item Distributions of the sum of the coordinates.
   \begin{equation}\label{sumpgf}
   G_{X_{1N} + X_{2N} + ... + X_{kN}}(z) = G_N (G_{Y_{11}, Y_{21}, ..., Y_{k1}}(z, z, ..., z)),
   \end{equation}
   \begin{equation}\label{pmfsum}
   P(X_{1N} + X_{2N} + ... + X_{kN} = s) = \frac{\partial^{s}G_N (G_{Y_{11}, Y_{21}, ..., Y_{k1}}(z, z, ..., z))}{s! \partial z^s} |_{z = 0}.
   \end{equation}
   \item Marginal distributions, given the sum of the coordinates. For $m = 0, 1, ..., s$,

   \medskip

   $P(X_{iN} = m| X_{1N} + X_{2N} + ... + X_{kN} = s) = $
   \begin{equation}\label{pmfgivensum}
   = \left(
        \begin{array}{c}
          s \\
          m \\
        \end{array}
      \right)  \frac{\frac{\partial^{s}G_N (G_{Y_{11}, ..., Y_{i-1 1}, Y_{i1}, Y_{i+1 1}, ..., Y_{k1}}(z, ..., z, z_i, z, ..., z))}{\partial z^{s-m} \partial z_i^m} |_{(z, z_i) = (0, 0)}}{\frac{\partial^{s}G_N (G_{Y_{11}, ..., Y_{k1}}(z, ..., z))}{\partial z^s} |_{z = 0}},
      \end{equation}
       $G_{X_{iN}}(z_i| X_{1N} + X_{2N} + ... + X_{kN} = s) =$
       $$ =  \left(
        \begin{array}{c}
          s \\
          m \\
        \end{array}
      \right) \frac{\sum_{m = 0}^s z_i^m \frac{\partial^{s}G_N (G_{Y_{11}, ..., Y_{i-1 1}, Y_{i1}, Y_{i+1 1}, ..., Y_{k1}}(z, ..., z, z_i, z, ..., z))}{\partial z^{s-m} \partial z_i^m} |_{(z, z_i) = (0, 0)}}{\frac{\partial^{s}G_N (G_{Y_{11}, ..., Y_{k1}}(z, ..., z))}{\partial z^s} |_{z = 0}}.$$

 \end{enumerate}
 \end{theorem}

\begin{proof}
\begin{enumerate}
  \item Equality (\ref{PGFGeneral}) follows by the definition of p.g.f. and the Fubini's theorem, in particular also known  as a "Double expectations formula". 
     Double application of equality (\ref{PGFGeneral}) for $(z_1, z_2, ..., z_k) = (0, 0, ..., 0)$, (see e.g. (34.2) in \cite{bib3}) leads us to (\ref{pmf0}).    The main properties of p.g.f., see e.g. p.327 in \cite{Denuitetal}, together with (\ref{PGFGeneral})  entails (\ref{pmf}).
   The total probability formula leads us to (\ref{pmffpv}).
      \item The property (\ref{pgfsubset}) follows by (\ref{PGFGeneral}) and the definition of p.g.f. See e.g. p. 3 in Johnson et. al \cite{bib3}. The proof of the (\ref{pmfsubsets}) is analogous to (\ref{pmf}).

      \item (\ref{pgfmarginal}) and $P(X_{iN} = 0) = G_N(P(Y_{i1} = 0))$ are particular cases, correspondingly of (\ref{pgfsubset}) and (\ref{pmfsubsets}), applied for $k = 1$. Equality (\ref{marginal}) is a consequence of the Total probability formula. (\ref{expectation}) and (\ref{veriance}) are Wald's equalities, applied for the coordinates in our model.  Formulae for double expectation and covariance entail (\ref{cov}).
       (\ref{cor}) is a consequence of the formula for correlation, (\ref{cov}) and (\ref{veriance}).

      \item Definition about conditional probability and (\ref{pmfsubsets}), entail (\ref{conditionals}).
      Replacement of (\ref{pmfsubsets}) and (\ref{pgfmarginal}) in formula (34.48), Johnson et. al \cite{bib3} gives (\ref{conditionalpgf}).
      One of the main properties of p.g.fs. is that its first derivative in $1$ is equal to the expectation of the r.v. ((34.15), Johnson et. al \cite{bib3}). Its application, together with (\ref{conditionalpgf}) gives (\ref{conditionalexpectation}).

      \item Formula for correlation, together with (\ref{covRaluca}) and (\ref{veriance}) prove (\ref{corelationX}).

      \item Following e.g. p.327 (\cite{Denuitetal}) we apply (\ref{PGFGeneral}) for $z_1 = z_2 = ... = z_k = z$ and obtain (\ref{sumpgf}).
             The main properties of p.g.f., see e.g. p.327 in \cite{Denuitetal}, together with (\ref{sumpgf})  entails (\ref{pmfsum}).

   \item Let $m = 0, 1, ..., s$.

   $P(X_{iN} = m, X_{1N} + ... + X_{i-1N} + X_{i+1N}... + X_{kN} = j) = $

  $$ = \frac{\partial^{m+j}G_{X_{iN}, X_{1N} + ... + X_{i-1N} + X_{i+1N}... + X_{kN}}(z_i, z)}{m! j!\partial z_i^m \partial z^j}|_{(z_i,z) = (0,0)} = $$
   $$= \frac{\partial^{m+j}G_N (G_{Y_{11}, ..., Y_{i-1 1}, Y_{i1}, Y_{i+1 1}, ..., Y_{k1}}(z, ..., z, z_i, z, ..., z))}{m! j!\partial z_i^m \partial z^j}|_{(z_i, z) = (0,0)}.$$

 The last equality, together with the definition for conditional probability and (\ref{pmfsum}), entail

 $P(X_{iN} = m| X_{1N} + X_{2N} + ... + X_{kN} = s) = $

 $$= \frac{P(X_{iN} = m, X_{1N} + ... + X_{i-1N} + X_{i+1N}... + X_{kN} = s - m) }{P(X_{1N} + ... + X_{i-1N} + X_{i+1N}... + X_{kN} = s - m)} =$$

   $$
   = \left(
        \begin{array}{c}
          s \\
          m \\
        \end{array}
      \right)  \frac{\frac{\partial^{s}G_N (G_{Y_{11}, ..., Y_{i-1 1}, Y_{i1}, Y_{i+1 1}, ..., Y_{k1}}(z, ..., z, z_i, z, ..., z))}{\partial z^{s-m} \partial z_i^m} |_{(z, z_i) = (0, 0)}}{\frac{\partial^{s}G_N (G_{Y_{11}, ..., Y_{k1}}(z, ..., z))}{\partial z^s} |_{z = 0}}.
     $$
       \end{enumerate}
\end{proof}

{\it Note:} 1. The second formula in (\ref{veriance}) shows that, the bigger the $FI\, N$, or $E\,Y_{i_1}$, or $FI\, Y_{i_1}$ are, the bigger the $FI\, X_{iN}$ is.

2. The strength of the linear dependence between $X_{iN}$ and $N$, expressed in (\ref{cor}), decreases with the increments of $FI\, N$ or $CV\, Y_{i1}$.

\section{Particular cases}

Here we apply the results from Section 2 in some particular cases of compounds with equal number of Mn or NMn distributed summands.
We  make a brief empirical study of their mean square regressions and observe that in any of these multivariate random sums it is close to linear.

\subsection{Multivariate compounds with Multinomial summands.}

Consider a partition $A_1, A_2, ..., A_k$ of the sample space $\Omega$. Mn distribution is the one of the numbers of outcomes $A_i$, $i = 1, 2, ..., k$ in a series of $n$ independent repetitions of a trial. More precisely, let $n \in \mathbb{N}$, $0 < p_i$, $i = 1, 2, ..., k$ and $p_1 + p_2 + ... + p_k = 1$. A vector $(\xi_1, \xi_2, ..., \xi_{k})$ is called Multinomially distributed with parameters $n, p_1, p_2, ..., p_k$, if its probability mass function (p.m.f.) is

$$P(\xi_1 = i_1, \xi_2 = i_2, ..., \xi_k = i_k) =  \left(\begin{array}{c}
                  n  \\
                  i_1, i_2, ..., i_k
                \end{array}\right) p_1^{i_1} p_2^{i_2} ... p_k^{i_k},$$
for $i_1 + i_2 + ... + i_k = n, \,\, i_s = 0, 1, ...,\,\, s = 1, 2, ..., k.$
Briefly $$(\xi_1, \xi_2, ..., \xi_{k}) \sim Mn (n; p_1, p_2, ..., p_k).$$
The distribution of any subset $\xi_{i_1}, \xi_{i_2}, ..., \xi_{i_r}$ of its coordinates, can be easily described using relation that
\begin{equation}\label{subsetMn}
(\xi_{i_1}, \xi_{i_2}, ..., \xi_{i_r}, n - (\xi_{i_1} + \xi_{i_2} + ... + \xi_{i_r})) \sim
\end{equation}
\hfill $\sim Mn(n; p_{i_1}, p_{i_2}, ..., p_{i_r}, 1 - (p_{i_1} + p_{i_2} + ... + p_{i_r})).$

A systematic investigation of this distribution, together with a very good list of references, could be found e.g. in Johnson et al. \cite{bib3}.
Considering this distribution for $k = 1$, we obtain Binomial(Bi) distribution.
For $n = 1$ the Mn distribution is  multivariate Bernoulli distribution. Some of the following results, for the particular case when $N$ is Poisson distributed, can be found in Smith \cite{Smith}. This distribution is partially investigated on in 2002, Daley et. al \cite{DaleyVJ}, p. 113.

 \begin{theorem}\label{Th2} Suppose $(Y_{1i}, Y_{2i}, ..., Y_{ki})$, $i = 1, 2, ...$ are independent Mn distributed random vectors with parameters $(s, p_1, p_2, ..., p_k)$, where $s \in \mathbb{N}$, $p_1 > 0,  p_2 > 0, ..., p_k > 0$ and $p_1 + p_2 + ... + p_k = 1.$ Denote the resulting distribution, defined in (\ref{Def1}) by $(X_{1N},  X_{2N}, ..., X_{kN}) \sim C\,N\,Mn(\vec{a}_N; s, p_1, p_2,..., p_k),$ where the vector $\vec{a}_N$ contains the parameters of the distribution of the r.v. $N$. $(X_{1N},  X_{2N}, ..., X_{kN})$ possesses the following properties.
\begin{enumerate}
  \item $G_{X_{1N}, X_{2N},..., X_{kN}}(z_1, z_2, ..., z_k) = G_{sN} (p_1z_1 + p_2 z_2 + ... + p_kz_k).$
  \item    For $x_1, x_2, ..., x_k = 0, 1, ...$, such that $x_1 + x_2 + ... + x_k$ is a multiple of $s$,
   $P(X_{1N} = x_1,  X_{2N} = x_2, ..., X_{kN} = x_k) = $
   \begin{equation}\label{pmfCMn}
   = \left(
         \begin{array}{c}
           x_1 + x_2 + ... + x_k \\
           x_1, x_2,... , x_k \\
         \end{array}
       \right)p_1^{x_1}p_2^{x_2}...p_k^{x_k}P(Ns = x_1 + x_2 + ... + x_k).
       \end{equation}
       $P(X_{1N} = 0,  X_{2N} = 0, ..., X_{kN} = 0) = P(N = 0).$
   \item For all $r = 1, 2, ..., k$ and for any subset of coordinates $X_{i_1N}, X_{i_2N}, ..., X_{i_rN}$,

  $G_{X_{i_1N}, X_{i_2N},..., X_{i_rN}}(z_{i_1}, z_{i_2}, ..., z_{i_r}) =$
  \begin{equation}\label{pgfsubsetMn}
  = G_{sN} (p_{i_1}z_{i_1} + p_{i_2}z_{i_2} + ... + p_{i_r}z_{i_r} + 1 - (p_{i_1} + ... + p_{i_r})).
  \end{equation}
  $P(X_{i_1N} = x_{i_1},  X_{i_2N} = x_{i_2}, ..., X_{i_rN} = x_{i_r}) = $
   \begin{equation}\label{pmfsubsetsMn}
   = \frac{p_{i_1}^{x_{i_1}}...p_{i_r}^{x_{i_r}}}{x_{i_1}!x_{i_2}!...x_{i_r}!}\sum_{m = 0}^{\infty} \frac{(m + x_{i_1} + x_{i_2} + ... + x_{i_r})!}{m!}.
   \end{equation}
\hfill    $.(1 - (p_{i_1} + ... + p_{i_r}))^m P[(sN = m + x_{i_1} + x_{i_2} + ... + x_{i_r}),$

  \bigskip

   $(X_{i_1N},  X_{i_2N}, ..., X_{i_rN}, Ns - (X_{i_1N} + X_{i_2N} + ... + X_{i_rN})) \sim$

  \medskip

\hfill    $\sim C\,N\, Mn(\vec{a}_N; s, p_{i_1}, p_{i_2}, ..., p_{i_r}, 1 - (p_{i_1} + ... + p_{i_r})).$

  \item For all $i = 1, 2, ..., k$, and $m = 0, 1, ...$,
    \begin{equation}\label{univariatepmfMn}
     P(X_{iN} = m) = \frac{p_i^m}{m!} \sum_{j = 0}^\infty \frac{(j + m)!}{j!}(1 - p_i)^jP(sN = j + m) =
     \end{equation}
    \begin{equation}\label{12}= \left\{\begin{array}{ccc}
                 G_{sN}(1 - p_i) & , & m = 0 \\
                 \sum_{n = 1}^\infty \left(
                                       \begin{array}{c}
                                         ns \\
                                         m \\
                                       \end{array}
                                     \right)p_i^m(1 - p_i)^{ns-m}P(N = n)   & , & m = 1, 2,...
               \end{array}
    \right.\end{equation}
   If $E\,N < \infty$, òî $E\,X_{i_N} = s p_i E\,N.$

   \medskip

   If $E\,(N^2) < \infty$, then  $Var\, X_{iN} = s p_i [s p_i Var\, N + (1 - p_i)E\,N]$ and
   \begin{equation}\label{FICMn}
   FI\, X_{iN}  = 1 + p_i (s FI\, N - 1).
   \end{equation}
      \item $P(X_{iN} = x_i| X_{jN} = x_j) = $
     \begin{equation}\label{conditionalMn}
     = \frac{p_i^{x_i}}{x_i!}\frac{ \sum_{m = 0}^\infty \frac{(m + x_i + x_j)!}{m!}(1 - (p_i + p_j))^m P(sN = m + x_i + x_j)} {\sum_{m = 0}^\infty \frac{(x_j + m)!}{m!} (1 - p_j)^m P(sN = x_j + m)},
     \end{equation}
   $ G_{X_{iN}}(z_i| X_{jN} = x_j) = $
   \begin{equation}\label{condpgfMn} = \frac{ \sum_{m = 0}^\infty \frac{(x_j + m)!}{m!}(p_iz_i + 1 - p_i - p_j)^mP(sN = x_j + m)}
   { \sum_{m = 0}^\infty \frac{(x_j + m)!}{m!}(1 - p_j)^mP(sN = x_j + m)}.
   \end{equation}
    \begin{equation}\label{ConditionalMeanMn}
     E (X_{iN}| X_{jN} = x_j) = p_i\frac{\sum_{m = 0}^\infty \frac{(x_j + m + 1)!}{m!}(1 - p_j)^mP(sN = x_j + m + 1)}{\sum_{m = 0}^\infty \frac{(x_j + m)!}{m!}(1 - p_j)^mP(sN = x_j + m)}
     \end{equation}

     $cov(X_{iN}, X_{jN}) = sp_ip_j[- E\, N + s\, Var\, N ]$

\medskip

    $cor(X_{iN}, X_{jN}) =$
    \begin{equation}\label{correlationCMn}= \frac{\sqrt{p_ip_j}(s Var\, N - EN)}{\sqrt{[sp_i Var\, N + EN(1 - p_i)][sp_j Var\, N + EN(1 - p_j)]}}=
    \end{equation}
    \begin{equation}\label{correlationCMnFI}= \sqrt{\frac{(FI\, X_{iN} - 1)(FI\, X_{jN} - 1)}{FI\, X_{iN} FI\, X_{jN} }}.\end{equation}
\item  If the following expectation exists, then
        \begin{equation}\label{MSRegression}
        E (X_{iN}| X_{jN} = x_j) = \frac{p_i}{p_j}(x_j + 1)\frac{P(X_{jN} = x_j + 1)}{P(X_{jN} = x_j)}.
        \end{equation}
 \item $G_{X_{1N} + X_{2N} + ... + X_{kN}}(z) = G_{sN} (z),$

      $P(X_{1N} + X_{2N} + ... + X_{kN} = m) = P(sN = m).$
\item For $j = 1, 2, ...$, $i = 1, 2, ..., k$,   $(X_{iN}| X_{1N} + X_{2N} + ... + X_{kN} = j) \sim Bi(j; p_i).$
   \end{enumerate}
\end{theorem}

\begin{proof}
\begin{enumerate}
\item  The form of p.g.f. follows by (\ref{PGFGeneral}) and the formula of the p.g.f. of Mn distribution. See e.g. (35.4) in \cite{bib3}.
\item This property is an immediate consequence of (\ref{pmf0}), (\ref{pmf}), the fact that $$X_{1N} + X_{2N} + ...+ X_{kN} = sN,$$ and relation between p.m.f. and p.g.f., see e.g. (34.2) in \cite{bib3}.
\item Here we replace the redundant variables in Equality 1), Th. \ref{Th2} with 1, use one of the main properties of p.g.f. (see e.g. p.327 in \cite{Denuitetal} or Daley et. al \cite{DaleyVJ}, p. 113) and obtain (\ref{pgfsubsetMn}). Using (\ref{subsetMn}) and the definition (\ref{Def1}), we have that

    $(X_{i_1N},  X_{i_2N}, ..., X_{i_rN}, Ns - (X_{i_1N} + X_{i_2N} + ... + X_{i_rN})) \sim $

    \hfill  $ \sim C\,N\,Mn(\vec{a}_N; s, p_{i_1}, p_{i_2}, ..., p_{i_r}, 1 - (p_{i_1} + ... + p_{i_r})).$

     Now we sum up over the possible values of the last coordinate of this vector, use  2), and obtain

 $P(X_{i_1N} = x_{i_1},  X_{i_2N} = x_{i_2}, ..., X_{i_rN} = x_{i_r}) =$
 $$ = \sum_{m = 0}^\infty P(X_{i_1N} = x_{i_1},  X_{i_2N} = x_{i_2}, ..., X_{i_rN} = x_{i_r}, Ns - (X_{i_1N} + X_{i_2N} + ... $$
 \hfill $ ... + X_{i_rN}) = m) = $
 $$ = \sum_{m = 0}^\infty \left(
                            \begin{array}{c}
                              x_{i_1} + x_{i_2} + ... + x_{i_r} + m\\
                              x_{i_1}, x_{i_2}, ..., x_{i_r}, m \\
                            \end{array}
                          \right)p_{i_1}^{x_{i_1}}...p_{i_r}^{x_{i_r}} [1 - (p_{i_1}+ ... +p_{i_r})]^m  .$$
\hfill                           $.P(Ns = x_{i_1} + x_{i_2} + ... + x_{i_r} + m) =$
   $$= \frac{p_{i_1}^{x_{i_1}}...p_{i_r}^{x_{i_r}}}{x_{i_1}!x_{i_2}!...x_{i_r}!}\sum_{m = 0}^{\infty} \frac{(m + x_{i_1} + x_{i_2} + ... + x_{i_r})!}{m!}. [1 - (p_{i_1} + ... + p_{i_r})]^m. $$
\hfill    $. P(sN = m + x_{i_1} + x_{i_2} + ... + x_{i_r}).$
   \item The first equality in (\ref{univariatepmfMn}) is a particular case of (\ref{pmfsubsetsMn}), for $r = 1$. The rest part of (\ref{univariatepmfMn}) is a consequence of the substitution $j + m = ns$ and the definition of p.g.f. The results about the numerical characteristics follow by the Wald's equalities and (\ref{subsetMn}). More precisely $E Y_{i1} = sp_i$ and $Var Y_{i1} = sp_i(1 - p_i)$.
\item The definitions for conditional probability and (\ref{pmfsubsetsMn}) imply (\ref{conditionalMn}).
By (\ref{conditionalpgf})

$G_{X_{iN}}(z_i| X_{jN} = x_j) =$
$$  =  \frac{\frac{\partial^{x_j}G_{sN} (p_iz_i + p_jz_j + 1 - p_i - p_j)}{\partial z_j^{x_j}}}
   {\frac{\partial^{x_j}G_{sN} (p_jz_j + 1 - p_j)}{\partial z_j^{x_j}}}|_{z_j = 0} =  \frac{p_j^{x_j} \left(\frac{\partial^{x_j}}{\partial y^{x_j}}G_{sN}(y)\right)|_{y = p_iz_i + 1 - p_i - p_j}}{p_j^{x_j} \left(\frac{\partial^{x_j}}{\partial y^{x_j}}G_{sN}(y)\right)|_{y = 1 - p_j}} =  $$
   $$ = \frac{\sum_{n = 0}^\infty \left(\frac{\partial^{x_j}}{\partial y^{x_j}}y^{sn}P(N=n)\right)|_{y = p_iz_i + 1 - p_i - p_j}}{\sum_{n = 0}^\infty\left(\frac{\partial^{x_j}}{\partial y^{x_j}}y^{sn}P(N = n)\right)|_{y = 1 - p_j}} = $$
 $$ = \frac{\sum_{n : sn \geq x_j}^\infty \frac{(sn)!}{(sn - x_j)!}(p_iz_i + 1 - p_i - p_j)^{sn-x_j}P(sN = sn)}{\sum_{n : sn \geq x_j}^\infty\frac{(sn)!}{(sn - x_j)!}(1 - p_j)^{sn - x_j}P(sN = sn)}.$$
 After the Substitution $m = sn - x_j$ we obtain (\ref{condpgfMn}).

In order to calculate covariation we use (\ref{covRaluca}) and the formula for covariance of Mn random vector $cov(Y_{i1}, Y_{j1}) = -np_ip_j$. See e.g. Johnson et. al \cite{bib3}. Together with $Var\, X_{iN} = s p_i [s p_i Var\, N + (1 - p_i)E\,N]$ and formula for the correlation they imply (\ref{correlationCMn}). Divide the numerator and denominator in $cor(X_{iN}, X_{jN})$ to $EN$ and come to
   $$ \frac{\sqrt{p_ip_j}(s FI\, N - 1)}{\sqrt{[sp_i FI\, N + 1 - p_i][sp_j FI\, N + 1 - p_j]}}.$$
   Compare this result with (\ref{FICMn}) we see that it is equal to (\ref{correlationCMnFI}).

   \item  By (\ref{ConditionalMeanMn}) and (\ref{univariatepmfMn}),

   $$E (X_{iN}| X_{jN} = x_j) = p_i\frac{\sum_{m = 0}^\infty \frac{(x_j + m + 1)!}{m!}(1 - p_j)^mP(sN = x_j + m + 1)}{\sum_{m = 0}^\infty \frac{(x_j + m)!}{m!}(1 - p_j)^mP(sN = x_j + m)} = $$
    $$ = p_i\frac{\frac{(x_j + 1)!}{p_j^{x_j + 1}}P(X_{jN} = x_j + 1)}{\frac{x_j!}{p_j^{x_j}}P(X_{jN} = x_j)} = \frac{p_i}{p_j}(x_j + 1)\frac{P(X_{jN} = x_j + 1)}{P(X_{jN} = x_j)}.$$
      \item These distributions are obtained by 1) Th. \ref{Th2}, (\ref{sumpgf}), (\ref{pmfsum}) and the properties of p.g.fs.

   \item By (\ref{Def1}), we have $(X_{iN}, X_{1N} + X_{2N} + ... + X_{i-1N} + X_{i+1N} + ... + X_{kN}) \sim$

   $\sim  C\,N\,Mn(\vec{a}_N; s, p_i, 1 - p_i)$.  This, together with the definitions for conditional probability, (\ref{pmfCMn}) and $Mn$ (via its p.m.f.) entail

  $P(X_{iN} = m| X_{1N} + X_{2N} + ... + X_{kN} = j) = $
  $$ = \frac{P(X_{iN} = m, X_{1N} +  ... + X_{i-1N} + X_{i+1N} + ... + X_{kN} = j - m)}{P(X_{1N} + X_{2N} + ... + X_{kN} = j)} =$$
  $$= \frac{\left(
             \begin{array}{c}
               j \\
               m \\
             \end{array}
           \right)
  p_i^{m} (1 - p_i)^{j-m} P(Ns = j)}{P(Ns = j)} =  \left(
             \begin{array}{c}
               j \\
               m \\
             \end{array}
           \right)
  p_i^{m} (1 - p_i)^{j-m},$$ \hfill $m = 0, 1, ..., j.$

  Now, the definition of $Bi$ distribution entails the desired result.
\end{enumerate}
\end{proof}

{\it Note:} 1. As corollaries of the above theorem, when we fix the distribution of $N$, we can obtain many new and well known distributions and their numerical characteristics and conditional distributions.

 \begin{center}
\scalebox{0.5}{\includegraphics[draft=false]{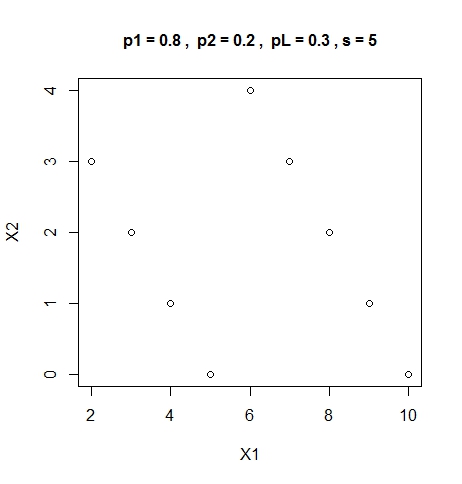}}

{\small Fig. 1.}
\end{center}

   \begin{center}
\scalebox{0.5}{\includegraphics[draft=false]{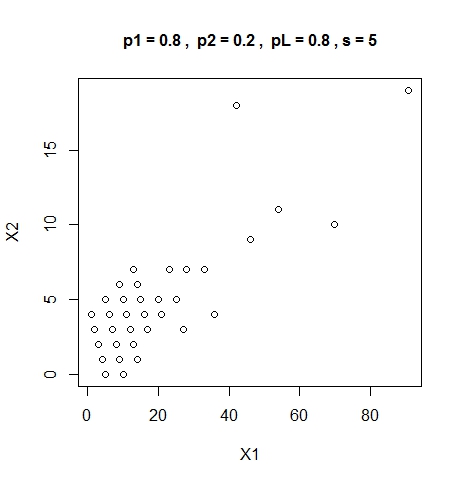}}

{\small Fig. 2. }
\end{center}

\begin{center}
\scalebox{0.5}{\includegraphics[draft=false]{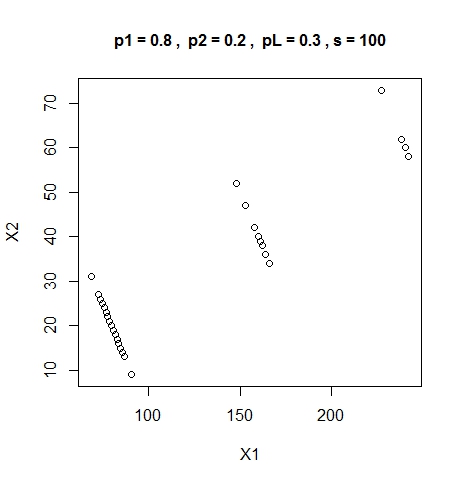}}

{\small Fig. 3.}
\end{center}

   \begin{center}
\scalebox{0.5}{\includegraphics[draft=false]{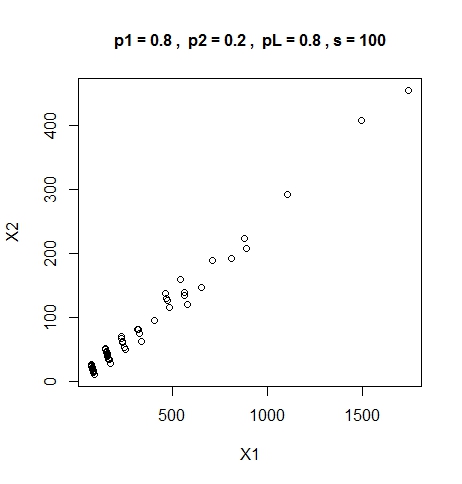}}

{\small Fig. 4. }
\end{center}

2. Here we denote by $LS(p_L)$ the Logarithmic series distribution with parameter  $p_L \in (0, 1)$. In order to observe the form of the dependence between the r.vs. $X_{1N}$, $X_{2N}$ and to model it via their mean square regression (\ref{MSRegression}), we simulate $100$ realizations of $(X_{1N}, X_{2N}) \sim CLSMn(p_L; s, p_1, p_2)$, for different values of the parameters $p_L$, $s$, $p_1$ and $p_2$. Four of the scatter plots of the data  are given on Figures 1 - 4. These observations show that the dependence between these two r.vs. is approximately linear, with slope almost $\frac{p_i}{p_j}$. The relation  (\ref{correlationCMnFI}) entails that the strength of the dependence between $X_{iN}$ and $X_{jN}$ is governed by $FI\, X_{iN}$ and $FI\, X_{jN}$.  When the last two Fisher indexes increase, the correlation between $X_{iN}$ and $X_{jN}$ also increases. Equation (\ref{FICMn}) entails that the bigger the $FI\,N$, or $s$ are, the bigger the $FI\, X_{iN}$ is. The last increases the  strength of the dependence between $X_{iN}$ and $X_{jN}$. All these results were observed empirically. More precisely Figures 1 - 4, show that for bigger $s$ or $p_L$ we obtain stronger linear dependence between $X_{iN}$ and $X_{jN}$.

3. The comparison of the coefficients before $P(N = n)$ in both expressions for the p.m.f. of $X_{iN}$, $i = 1, 2, ..., k$, (the first obtained using the Total probability formula and the second is (\ref{12})) gives us another proof of the equality $$\sum_{j_1 + j_2 + ... + j_n = m} \left(
                                       \begin{array}{c}
                                         s \\
                                         j_1 \\
                                       \end{array}
                                     \right)\left(
                                       \begin{array}{c}
                                         s \\
                                         j_2 \\
                                       \end{array}
                                     \right)...\left(
                                       \begin{array}{c}
                                         s \\
                                         j_n \\
                                       \end{array}
                                     \right) = \left(
                                       \begin{array}{c}
                                         ns \\
                                         m \\
                                       \end{array}
                                     \right).$$
In analogous way, working with the theoretical results, related with multivariate compounds with equal number of summands we can obtain new similar equalities.

\subsection{Compounds with NMn summands}

 In this subsection we consider the case, when the summands in (\ref{Def1}) are NMn. It seems that in 1952, Bates et. al \cite{Bates} introduced NMn distribution. They reached it considering Poisson distribution with Gamma mixing variable. Their first parameter could be a real number. Wishart \cite{Wishart} investigated the case, when the first parameter could be only integer. He called this distribution Pascal multinomial distribution.   From practical point of view, the NMn distribution  is interpreted as the one of the numbers of outcomes $A_i$, $i = 1, 2, ..., k$ before the $n$-th $B$, in series of independent repetitions, where $A_i$, $i = 1, 2, ..., k$ and $B$ form a partition of the sample space.  Let us remind the definition.

 Let $n \in \mathbb{N}$, $0 < p_i$, $i = 1, 2, ..., k$ and $p_1 + p_2 + ... + p_k < 1$. A vector $(\xi_1, \xi_2, ..., \xi_{k})$ is called Negative multinomially distributed with parameters $n, p_1, p_2, ..., p_k$, if its probability mass function (p.m.f.) is

$P(\xi_1 = i_1, \xi_2 = i_2, ..., \xi_k = i_k) = $
$$= \left(\begin{array}{c}
                  n + i_1 + i_2 + ... i_k -1 \\
                  i_1, i_2, ..., i_k,n-1
                \end{array}\right) p_1^{i_1} p_2^{i_2} ... p_k^{i_k} (1  - p_1 - p_2 - ... - p_k)^n,$$
$i_s = 0, 1, ..., s = 1, 2, ..., k.$
Briefly $(\xi_1, \xi_2, ..., \xi_{k}) \sim NMn (n; p_1, p_2, ..., p_k)$.

We will use the distribution of its coordinates. For $m  = 2, 3, ..., k-1$,
\begin{equation}\label{subsetNMn}
(\xi_{i_1}, \xi_{i_2}, ..., \xi_{i_r}) \sim NMn (s; \rho_{i_1}, \rho_{i_2}, ..., \rho_{i_r}),
\end{equation} with $\rho_{i_m} = \frac{p_{i_m}}{1 - \sum_{j \not \in \{i_1, i_2, ..., i_r\}}p_j}$, $ m = 1, 2, ..., r.$
More properties of NMn distribution can be found in Johnson et al. \cite{bib3}.

If $A_1, A_2, ..., A_k$ describe all possible mutually exclusive "successes" and the event $\overline{A}_1 \cap \overline{A}_2 \cap ... \cap \overline{A}_k$ presents the "failure", then the coordinates $\xi_i$ of the above vector, can be interpreted as the number of "successes" of type $A_i$ $i = 1, 2, ..., k$ until n-th "failure". Considering this distribution for $k = 1$, we obtain $NBi$ distribution with parameters $s$ and $(1 - p_1)$.
Let us note that if $k = 2, 3, ...$ the marginal distributions of NMn distributed random vector, defined above are $NBi(n, 1 - \rho_i)$, $\rho_{i} = \frac{p_{i}}{1 - \sum_{j \not = i}p_j}$. More precisely their p.g.f. is
$$G_\xi(z) = Ez^{\xi_{i}} = \left(\frac{1-\rho_i}{1 - \rho_iz}\right)^n, \,\,|z| < \frac{1}{\rho_i}\,\, i = 1, 2, ..., k.$$
Having in mind the univariate distributions, the NMn distribution is sometimes called Multivariate Negative Binomial distribution.
For $n = 1$ the NMn distribution is  Multivariate geometric distribution. Some properties of the bivariate version of this distribution are considered e.g. in 1981, by Phatak et al. \cite{Phatak}. A systematic investigation of multivariate version could be found e.g. in Srivastava et al. \cite{Srivastava}.
Here we consider random sums with Mn distributed summands. The case when $N$ is Poisson distributed is considered in Smith \cite{Smith}. The next theorem generalizes his results.

 \begin{theorem} \label{Th3} Suppose $(Y_{1i}, Y_{2i}, ..., Y_{ki}) \sim NMn(s, p_1, p_2, ..., p_k)$, $i = 1, 2, ...$ are independent. Here $s \in \mathbb{N}$, $p_1 > 0,  p_2 > 0, ..., p_k > 0$ and $$p_1 + p_2 + ... + p_k < 1.$$ Denote by $p_0 = 1 - p_1 - p_2 - ... - p_k$. Then the resulting distribution of
$$(X_{1N},  X_{2N}, ..., X_{kN}) \sim C\, N\, NMn(\vec{a}_N; s, p_1, p_2,..., p_k),$$
defined in (\ref{Def1}), possesses the following properties.

\begin{enumerate}
  \item Its p.g.f. is
  \begin{equation}\label{pgfCNMn}
  G_{X_{1N}, X_{2N},..., X_{kN}}(z_1, z_2, ..., z_k) = G_{sN} \left(\frac{p_0}{1 - p_1z_1 - p_2 z_2 - ... - p_kz_k}\right).
  \end{equation}
  \item The p.m.f. satisfies the equalities:

  $P(X_{1N} = 0,  X_{2N} = 0,  …, X_{kN} = 0) = G_{sN}(p_0).$

\medskip

   For $s = 1, 2, .., k$, $x_s = 0, 1, ...$: $(x_1, x_2, ..., x_k) \not= (0, 0, ..., 0),$

\medskip

$P(X_{1N} = x_1,  X_{2N} = x_2, ..., X_{kN} = x_k) = $

   \begin{equation}\label{pmfCNMnnot0} = p_1^{x_1}p_2^{x_2}...p_k^{x_k} \sum_{j = 1}^\infty \left(
         \begin{array}{c}
           sj + x_1 + x_2 + ... + x_k - 1\\
           x_1, x_2,... , x_k, sj - 1 \\
         \end{array}
       \right)p_0^{js}P(N = j).
       \end{equation}
   \item For all $r = 1, 2, ..., k$ and for all $X_{i_1N}, X_{i_2N}, ..., X_{i_rN}$,

$G_{X_{i_1N}, X_{i_2N},..., X_{i_rN}}(z_{i_1}, z_{i_2}, ..., z_{i_r}) = $
 \begin{equation}\label{pgfsubsetNMn}
   = G_{sN}\left(\frac{\rho_{0r}}{1 - \rho_{i_1}z_{i_1} - \rho_{i_2} z_{i_2} - ... - \rho_{i_r}z_{i_r}}\right),
  \end{equation}
  where $\rho_{i_j} = \frac{p_{i_j}}{p_0 + \sum_{j = 1}^r p_{i_j}}, \quad \rho_{0r} = 1 - \rho_{i_1} - ... - \rho_{i_r},$
  i.e. $$(X_{i_1N},  X_{i_2N}, ..., X_{i_rN}) \sim C\,N\, NMn(\vec{a}_N; s, \rho_{i_1}, \rho_{i_2},..., \rho_{i_r}).$$

  $P(X_{i_1N} = 0,  X_{i_2N} = 0,  …, X_{i_rN} = 0) = G_{sN}(\rho_{0r}).$

\bigskip

   For $s = 1, ..., r$, $x_{i_s} = 0, 1, ...$: $(x_{i_1}, x_{i_2}, ..., x_{i_r}) \not= (0, 0, ..., 0),$

\bigskip

   $P(X_{i_1N} = x_{i_1},  X_{i_2N} = x_{i_2}, ..., X_{i_rN} = x_{i_r}) = $
    \begin{equation}\label{pmfsubsetsNMn}
    = \rho_{i_1}^{x_{i_1}}... \rho_{i_r}^{x_{i_r}}\sum_{j = 1}^{\infty} \left(
                                                                            \begin{array}{c}
                                                                              sj + x_{i_1} + x_{i_2} + ... + x_{i_r} - 1\\
                                                                              x_{i_1}, x_{i_2}, ..., x_{i_r} + sj - 1 \\
                                                                            \end{array}
                                                                          \right)\rho_{0r}^{js}P(N = j).                                                                         \end{equation}
   \item    If $E\,N < \infty$, then $E\,X_{i_N} = s \frac{p_i}{p_0} E\,N.$

   \medskip

   If $E\,(N^2) < \infty$, then  $Var\, X_{iN} = s \frac{p_i^2}{p_0^2} [s Var\, N + (1 + \frac{p_0}{p_i})E\,N],$
   \begin{equation}\label{FICNMn}
   FI\, X_{iN}  = 1 + \frac{p_i}{p_0}[s FI\, N +  1].
   \end{equation}
      \item Dependence between the coordinates.
    \begin{equation}\label{corelationXNMn}
    cor(X_{iN}, X_{jN}) = \sqrt{\frac{(FI\, X_{iN} - 1)(FI\, X_{jN} - 1)}{FI\, X_{iN} FI\, X_{jN} }}.
    \end{equation}
  \item Its conditional distributions satisfy the next equalities.

  For $x_j \not= 0$, $x_i = 0, 1, ...$,
       \begin{equation}\label{conditionalNMn}
       P(X_{iN} = x_i| X_{jN} = x_j) = \frac{1}{x_i!}\left( \frac{p_0 + p_j}{p_0 + p_i + p_j} \right)^{x_j} \left( \frac{p_i}{p_0 + p_i + p_j} \right)^{x_i}.
       \end{equation}
    $$ .\frac{ \sum_{n = 1}^\infty \frac{(sn + x_i + x_j - 1)!}{(sn - 1)!}\left( \frac{p_0}{p_0 + p_i + p_j} \right)^{sn} P(N = n)} {\sum_{n = 1}^\infty \frac{(sn + x_j - 1)!}{(sn - 1)!}\left( \frac{p_0}{p_0 + p_j} \right)^{sn} P(N = n)}.$$
   For $x_i = 1, 2, ...$, $P(X_{iN} = x_i| X_{jN} = 0) =$
    $$ = \frac{\left( \frac{p_i}{p_0 + p_i + p_j} \right)^{x_i}}{G_{sN}\left( \frac{p_0}{p_0 + p_j} \right)}\sum_{n = 1}^\infty \frac{(sn + x_i - 1)!}{x_i!(sn - 1)!}\left( \frac{p_0}{p_0 + p_i + p_j} \right)^{sn} P(N = n),$$
         $$P(X_{iN} = 0| X_{jN} = 0) = \frac{G_{sN}\left( \frac{p_0}{p_0 + p_i + p_j} \right)}{G_{sN}\left( \frac{p_0}{p_0 + p_j} \right)}.$$
\item  For $x_j = 1, 2, ...$
\begin{equation}\label{MSRegressionNMn}
        E (X_{iN}| X_{jN} = x_j) = \frac{p_i}{p_j}(x_j + 1)\frac{P(X_{jN} = x_j + 1)}{P(X_{jN} = x_j)}.
        \end{equation}
   \item  $X_{1N} + X_{2N} + ... + X_{kN} \sim CNNMn(s; 1 - p_0).$
   \item For $j = 1, 2, ...$, $i = 1, 2, ..., k$,    $$(X_{iN}| X_{1N} + X_{2N} + ... + X_{kN} = j) \sim Bi(j; \frac{p_i}{1 - p_0}).$$
      \end{enumerate}
\end{theorem}

\begin{proof}
\begin{enumerate}
\item  The equality (\ref{pgfCNMn}) follows by (\ref{PGFGeneral}) and the formula of the p.g.f. of NMn distribution. See e.g. (36.1) in \cite{bib3}, who use different parametrisation.
\item The formula of the p.m.f. at $(0, 0, ..., 0)$ is an immediate consequence of (\ref{pmf0}), the fact that $G_N(z^s) = G_{sN}(z)$ and the form of p.m.f. of $(Y_{11}, Y_{21}, ..., Y_{k1}) \sim NMn(s, p_1, p_2, ..., p_k)$ at $(0, 0, ..., 0)$.

    To prove (\ref{pmfCNMnnot0}) we can use the Total probability formula (the partition of the sample space describe the possible values of $N$), and the fact that $n$-th fold convolution on $NMn(s, p_1, p_2, ..., p_k)$ is $NMn(ns, p_1, p_2, ..., p_k)$. The same formula can be obtained also by (\ref{pmf}) and relation between p.m.f. and p.g.f.
\item Equality 1), Th. \ref{Th3} entails (\ref{pgfsubsetNMn}), when replace redundant variables with 1 and use one of the main properties of p.g.f., see e.g. p.327 in \cite{Denuitetal}. Using (\ref{subsetNMn}) and the definition (\ref{Def1}), we have that

    $$(X_{i_1N},  X_{i_2N}, ..., X_{i_rN}) \sim C\,N\,NMn(\vec{a}_N; s, \rho_{i_1}, \rho_{i_2}, ..., \rho_{i_r}).$$

     Now we use (\ref{pmfsubsetsNMn}) and complete the proof.
\item The substitution of the mean $E Y_{i1} = s\frac{\rho_i}{1 - \rho_{i}} = s\frac{p_i}{p_0}$,  and the variance  $Var Y_{i1} = s\frac{p_i(p_0 + p_i)}{p_0^2}$, $i = 1, 2, ..., k,$ of NMn distribution in (\ref{expectation}) and (\ref{veriance}) gives us $E\,X_{i_N}$ and $Var\,X_{i_N}$, $i = 1, 2, ..., k$.
\item Using the properties of $NMn$ distribution we have that
$$[CV\,(Y_{i1}, Y_{i2})]^2 = \frac{1}{s}, \,(CV\, Y_{i1})^2 = \frac{1}{s}(1 + \frac{p_0}{p_i}),\, (CV\, Y_{j1})^2 = s^{-1}(1 + \frac{p_0}{p_j}).$$
  Now the equation (\ref{corelationX}), together with (\ref{FICNMn}) entail the desired result.
\item The definitions for conditional probability and (\ref{pmfsubsetsNMn}), applied for $r = 1$ and $r = 2$, imply (\ref{conditionalNMn}).
In analogous way we prove the case when $X_{jN} = 0$.

\item Let $x_j$ be a fixed positive integer.  In the next calculations we first use the formula for the mean, then we change the order of summation, use the derivatives of the geometric series, make some algebra and obtain.
$$E (X_{iN}| X_{jN} = x_j) = \sum_{x_j = 1}^\infty \frac{1}{(x_i - 1)!}\left( \frac{p_0 + p_j}{p_0 + p_i + p_j} \right)^{x_j} \left( \frac{p_i}{p_0 + p_i + p_j} \right)^{x_i}.$$
    $$ .\frac{ \sum_{n = 1}^\infty \frac{(sn + x_i + x_j - 1)!}{(sn - 1)!}\left( \frac{p_0}{p_0 + p_i + p_j} \right)^{sn} P(N = n)} {\sum_{n = 1}^\infty \frac{(sn + x_j - 1)!}{(sn - 1)!}\left( \frac{p_0}{p_0 + p_j} \right)^{sn} P(N = n)} = $$
$$ = \left( \frac{p_i}{p_0 + p_i + p_j} \right)\sum_{n = 1}^\infty \frac{P(N = n)}{(sn - 1)!}\left( \frac{p_0}{p_0 + p_i + p_j} \right)^{sn}\left( \frac{p_0 + p_j}{p_0 + p_i + p_j} \right)^{x_j} .$$
    $$ .\frac{ \sum_{x_j = 1}^\infty \frac{(sn + x_i + x_j - 1)!}{(x_i - 1)!}\left( \frac{p_i}{p_0 + p_i + p_j} \right)^{x_i - 1}}{\sum_{n = 1}^\infty \frac{(sn + x_j - 1)!}{(sn - 1)!}\left( \frac{p_0}{p_0 + p_j} \right)^{sn} P(N = n)} = $$
$$ = \left( \frac{p_i}{p_0 + p_i + p_j} \right)\sum_{n = 1}^\infty \frac{P(N = n)}{(sn - 1)!}\left( \frac{p_0}{p_0 + p_i + p_j} \right)^{sn}\left( \frac{p_0 + p_j}{p_0 + p_i + p_j} \right)^{x_j} .$$
    $$ .\frac{ \frac{(sn + x_j)!}{\left(1 -  \frac{p_i}{p_0 + p_i + p_j} \right)^{sn + x_j + 1}}}{\sum_{n = 1}^\infty \frac{(sn + x_j - 1)!}{(sn - 1)!}\left( \frac{p_0}{p_0 + p_j} \right)^{sn} P(N = n)} = $$
$$ = (x_j + 1)\frac{p_i}{p_j}\frac{\left( \frac{p_j}{p_0 + p_j} \right)^{x_j + 1}\sum_{n = 1}^\infty \frac{P(N = n)(sn + x_j)!}{(sn - 1)! (x_j + 1)!}\left( \frac{p_0}{p_0 + p_j} \right)^{sn}}{\left(\frac{p_j}{p_0 + p_j}\right)^{x_j}\sum_{n = 1}^\infty \frac{(sn + x_j - 1)!}{(sn - 1)!x_j!}\left( \frac{p_0}{p_0 + p_j} \right)^{sn} P(N = n)} = $$
    $$ = \frac{p_i}{p_j}(x_j + 1)\frac{P(X_{jN} = x_j + 1)}{P(X_{jN} = x_j)}.$$
Finally we have just compared the previous expression with $P(X_{jN} = x_j + 1)$ and $P(X_{jN} = x_j)$, using (\ref{pmfsubsetsNMn}), for $r = 1$.
\item These distributions are obtained by (\ref{sumpgf}), 1) Th. \ref{Th3} for $z_1 = z_2 = ... = z_k = z$ and (\ref{pgfsubsetNMn}), for $r = 1$.
\item By (\ref{Def1}) and the definition of the NMn distribution, we have

   $(X_{iN}, X_{1N} + X_{2N} + ... + X_{i-1N} + X_{i+1N} + ... + X_{kN}) \sim$

  \hfill $\sim  C\,N\,Mn(\vec{a}_N; s, p_i, 1 - p_0 - p_i)$.

  This, together with the definitions for conditional probability, (\ref{pmfCMn}) and $Mn$ (via its p.m.f.) entail

  $P(X_{iN} = m| X_{1N} + X_{2N} + ... + X_{kN} = j) = $
  $$ = \frac{P(X_{iN} = m, X_{1N} +  ... + X_{i-1N} + X_{i+1N} + ... + X_{kN} = j - m)}{P(X_{1N} + X_{2N} + ... + X_{kN} = j)} =$$
  $$= \frac{p_i^m (1 - p_0 - p_i)^{j-m} \sum_{n = 1}^{\infty} \left(\begin{array}{c} sn + j - 1\\ m, j - m, sn-1 \\ \end{array} \right)p_0^{ns}P(N = n).}{(1 - p_0)^j \sum_{n = 1}^{\infty} \left( \begin{array}{c} sn + j - 1\\ j \\ \end{array}\right)p_0^{ns}P(N = n).} = $$
  $$ =  \left(\begin{array}{c}
               j \\
               m \\
             \end{array}
           \right)  \left(\frac{p_i}{1 - p_0}\right)^{m} (1 - \frac{p_i}{1 - p_0})^{j},\quad m = 0, 1, ..., j.$$
             Now we use the definition of $Bi$ distribution and complete the proof.
\end{enumerate}
\end{proof}

{\it Note:}  1.   In order to observe the form of the dependence between the r.vs. $X_{1N}$, $X_{2N}$, we have simulated $100$ realizations of $(X_{1N}, X_{2N}) \sim$ $C\,LS\,NMn(p_L;$ $s, p_1,$ $p_2)$ and have plotted the scatter plots of the observations. Here $LS(p_L)$ means the Logarithmic series distribution with parameter  $p_L \in (0, 1)$. Four examples of such scatter plots, for different values of the parameters $p_L$, $s$, $p_1$ and $p_2$, are given on Figures 5 - 8. They show that the dependence between these two r.vs. is approximately linear, with slope, close to $\frac{p_i}{p_j}$. The strength of the dependence increases with parameters $s$ and $p_L$. These empirical observations are synchronised with the theoretical relations, expressed in the equations (\ref{MSRegressionNMn}), (\ref{corelationXNMn}) and (\ref{FICNMn}).

  \begin{center}
\scalebox{0.5}{\includegraphics[draft=false]{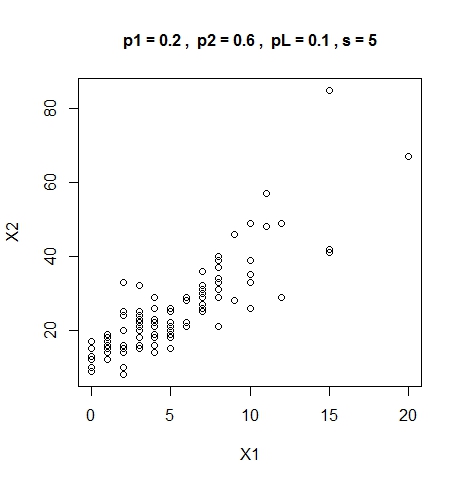}}

{\small Fig. 5. }
\end{center}

 \begin{center}
\scalebox{0.5}{\includegraphics[draft=false]{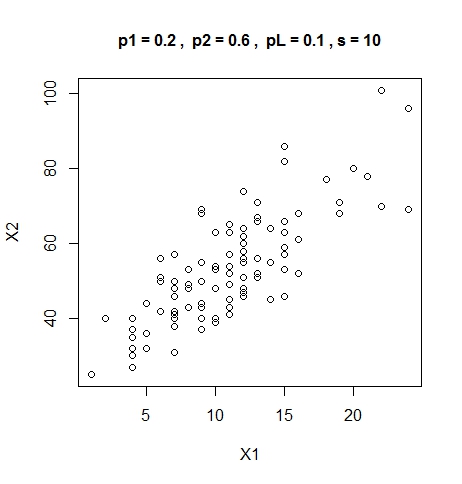}}

{\small Fig. 6. }
\end{center}

   \begin{center}
\scalebox{0.5}{\includegraphics[draft=false]{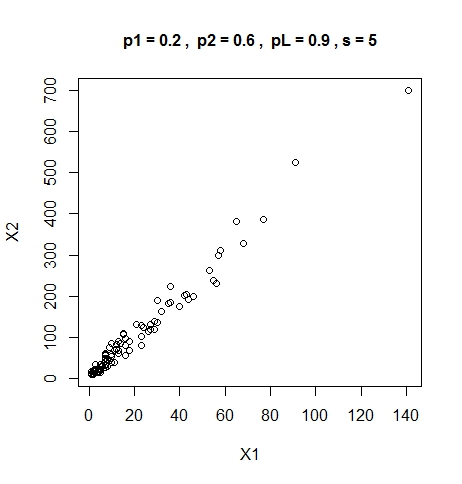}}

{\small Fig. 7. }
\end{center}

 \begin{center}
\scalebox{0.5}{\includegraphics[draft=false]{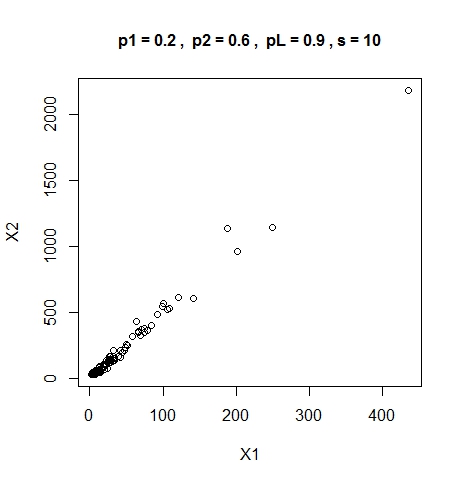}}

{\small Fig. 8. }
\end{center}

2. Similarly to the previous case, the relation  (\ref{corelationXNMn}) shows that the strength of the linear dependence between $X_{iN}$ and $X_{jN}$ is completely governed by the values of $FI\, X_{iN}$ and $FI\, X_{jN}$.  When the last two Fisher indexes increase, the correlation between $X_{iN}$ and $X_{jN}$ also increases. From the other side, the equation (\ref{FICNMn}) entails that the bigger the $FI\,N$, or $s$, or $p_0$, or $p_i$ are, the bigger the $FI\, X_{iN}$ is. So, via $FI\, X_{iN}$, the last increases the  strength of the straight-line dependence between $X_{iN}$ and $X_{jN}$. On the examples on Figure 3 and Figure 4 we observe that the bigger the s, the stronger the linear dependence is.

\section{Conclusive remarks}

The paper considers the properties of the distributions of multivariate random sums with equal number of summands. Some simulations of realizations of particular cases of random vectors with such distributions are made, and the dependence between their coordinates is observed. These investigations show that these distributions are appropriate for modelling of both: linear dependence between the coordinates and clustering in the observations. From theoretical point of view, in both particular cases, described here, these results can be seen in equalities (\ref{MSRegression}) and (\ref{MSRegressionNMn}). For the general case, without fixing neither the distribution of the number of summands, nor the distribution of summands, the second formula in (\ref{veriance}) shows that, the bigger the Fisher index(FI) of the number of summands, or the expectation, or the FI of the summands are, the bigger the FI of the corresponding random sum is. The strength of the linear dependence between the coordinates $X_{iN}$ of the random sums and the number of summands $N$, expressed in (\ref{cor}), decreases with the increments of the FI of $N$ or the coefficient of variation of the summands.

If  $(X_{1N}, X_{2N}) \sim C\, LS\, Mn (p_L; s, p_1, p_2)$ the dependence between these two r.vs. is approximately linear, with slope almost $\frac{p_i}{p_j}$. The relation  (\ref{correlationCMnFI}) entails that the strength of the dependence between $X_{iN}$ and $X_{jN}$ is governed by $FI\, X_{iN}$ and $FI\, X_{jN}$.  When the last two Fisher indexes increase, the correlation between $X_{iN}$ and $X_{jN}$ also increases. Equation (\ref{FICMn}) entails that the bigger the $FI\,N$, or $s$ are, the bigger the $FI\, X_{iN}$ is. Figures 1 - 4 show that for bigger parameters $s$ or $p_L$ we obtain stronger linear dependence between $X_{iN}$ and $X_{jN}$.

It is interesting that in the other, considered particular case here, i.e. if  $(X_{1N}, X_{2N}) \sim C\, LS\, NMn (p_L; s, p_1, p_2)$, the form of the dependence between $X_{1N}$ and $X_{2N}$ is again approximately linear, with slope almost $\frac{p_i}{p_j}$, for all possible discrete distributions of $N$. This is synchronised with (\ref{MSRegressionNMn}). The strength of the dependence increases with $s$ and $p_L$. The equation (\ref{FICNMn}) entails that the bigger the $FI\,N$, or $s$, or $p_0$, or $p_i$ are, the bigger the $FI\, X_{iN}$ is. Similarly to the previous case, the relation  (\ref{corelationXNMn}) shows that the strength of the linear dependence between $X_{iN}$ and $X_{jN}$ is completely governed by the values of $FI\, X_{iN}$ and $FI\, X_{jN}$.  When the last two Fisher indexes increase, the correlation between $X_{iN}$ and $X_{jN}$ also increases.

\bigskip

\normalsize \noindent \sl
Pavlina Kalcheva Jordanova

Shumen University

115 Universitetska str.,

9700 Shumen, Bulgaria

e-mail: {\tt pavlina$\_$kj@abv.bg}

\vfill \eject

{\hspace{1pt}} \label{end}

\end{document}